\documentclass[11pt]{article}

\usepackage[margin=1in]{geometry}
\usepackage{microtype}

\usepackage{amsmath,amssymb,amsthm,mathtools}

\usepackage{array}
\usepackage{booktabs}
\usepackage{float}
\usepackage{graphicx}
\usepackage{tikz}
\usetikzlibrary{arrows.meta,calc,positioning}

\usepackage[hidelinks]{hyperref}

\newtheorem{theorem}{Theorem}[section]
\newtheorem{proposition}[theorem]{Proposition}
\newtheorem{corollary}[theorem]{Corollary}
\newtheorem{lemma}[theorem]{Lemma}
\theoremstyle{definition}
\newtheorem{definition}[theorem]{Definition}
\newtheorem{remark}[theorem]{Remark}

\newcommand{\Z}{\mathbb{Z}}
\newcommand{\Conj}{\operatorname{Conj}}
\newcommand{\HomQnd}{\operatorname{Hom}_{\mathrm{Qnd}}}
\newcommand{\HomZ}{\operatorname{Hom}_{\Z}}
\newcommand{\rowspan}{\operatorname{rowspan}}

\title{Signature Modules and the Dihedral\\
Conjugation-Quandle Counting Invariant}
\author{Zining Fan\\Emmaus High School}
\date{July 2026}

\begin{document}

\maketitle

\begin{abstract}
We use reflection and rotation signatures of dihedral groups to develop a Smith normal form method for computing the conjugation quandle colorings across all dihedral groups. Let $L=L_1\cup\cdots\cup L_\ell$ be a link with link components $L_1,\ldots,L_\ell$, and let $\Conj(D_n)$ be the conjugation quandle of the dihedral group $D_n$. We use the semidirect-product structure $D_n\cong\Z_n\rtimes\Z_2$, to separate the coloring into a component signature. This component signature tells us whether each link component
is colored by rotations or reflections, and the exponents
are assigned separately under modulus $n$. Once the component signature is determined, every crossing relation becomes an equation that determines an integer matrix $M_\tau$.

The abelian group $A_\tau(L)$, which we call the signature module, is represented by an integer matrix $M_\tau$. We prove that the free rank and nonunit Smith normal form entries, or equivalently the isomorphism class of $A_\tau(L)$, are preserved under the Reidemeister moves. The complete family of signature modules also determines the entire family of counting invariants for all $D_n$. The all-reflection system is the same as the Fox--$n$ colorings, while the mixed rotation-reflection signatures contain information that gives a stronger invariant. We also show that each signature module is the $t_k=(-1)^{\tau_k}$ specialization of the multivariable
Alexander module, and extend the coloring interpretation of the module to generalized dihedral groups $\operatorname{Dih}(B)$
for every abelian group $B$.
\end{abstract}


\section{Introduction}

A knot is an embedding of $S^1$ into $\mathbb{R}^3$, or a circle into 3-dimensional space. Knots are considered equivalent up to ambient isotopy. A collection of disjoint embedded circles is considered a link. Knots and links are represented as projections onto $\mathbb{R}^2$ with a knot diagram. Two diagrams represent the same link if and only if they are related by a finite sequence of Reidemeister moves~\cite[pp.~1--4]{lickorish}.

A knot or link invariant is a quantity assigned to the knot or link that remains unchanged under Reidemeister moves. The fundamental quandle is one example of link invariants \cite{joyce,matveev}. The fundamental quandle is a complete knot invariant up to mirror image~\cite[pp.~103--104]{kamada2002}, however it is very difficult to compute. A more computable invariant is obtained by choosing a finite quandle $T$ and counting homomorphisms~\cite[p.~104]{kamada2002}
\[
\left|\operatorname{Hom}(Q(L),T)\right|.
\]

This paper explores the conjugation quandle over the full family of dihedral groups. Every element of $D_n$ is either a rotation $r^x$ or a reflection $sr^x$~\cite[Section~1.2]{dummitfoote}. This separates the colorings into a binary component signature of a reflection or rotation, and an exponent system with a modulus $n$. Once the signature is determined, we can use a matrix to record the crossings of the knot or link. We then study this matrix using Smith normal form.

Smith normal forms of Fox~\cite[Section~2.1]{ge} and linear Alexander coloring matrices~\cite[Sections~1 and~2.1]{kauffman} are already used in knot theory. In the Fox coloring, the nonunit Smith normal form of the coloring matrix is a link invariant~\cite[Section~2.1]{ge}. The distinction between the reflection-rotation signature and the Fox-$n$ coloring is that the reflection-rotation signature contains more information about the link. The all-reflection signature recovers the Fox colorings, but links with multiple components also contain mixed signatures not covered under the Fox colorings~\cite[Remark~2 and Definition~1]{IIMS2025}.

Rotation--reflection coloring relations were studied in
\cite{IIMS2025}, and the multivariate Alexander colorings and
their specializations were developed in \cite{traldi2018}.
However, \cite{IIMS2025} focuses on two-tone colorability and
surjective dihedral representations, while this paper gives
an explicit connection between dihedral conjugation-quandle
colorings and multivariate Alexander colorings. This paper
builds on the results of the former papers by Traldi and IIMS,
and determines the dihedral conjugation-quandle counting
invariants for all moduli at once. In this paper, we organize
the crossing relations into a family of signature modules
$A_\tau(L)$ based on rotation-reflection signatures and give
an integer-matrix proof that their isomorphism classes stay
invariant under Reidemeister moves. Their Smith normal forms
then determine the $\operatorname{Conj}(D_n)$ counting invariant
for every $n\geq3$, without needing to know the result of each
modulus independently.

We also prove that the full family of signature modules is
stronger than the Fox coloring counts. We show that it
determines the number of Fox $n$-colorings for every $n$, and
give a specific example where the full family of signature
modules is able to detect the difference between two links,
but the Fox coloring counts are not. We show that the
signature modules are the $t_k=(-1)^{\tau_k}$ specializations
of the multivariable Alexander module and extend the coloring
interpretation to generalized dihedral groups
$\operatorname{Dih}(B)$ for every abelian group $B$.

The main contribution of this paper is the family of signature modules $A_\tau(L)$. We prove that each module does not depend on the chosen diagram. We also demonstrate that the complete family determines the $\Conj(D_n)$ counting invariant for every dihedral group at the same time. Thus, this result is not simply a faster way to count solutions for the conjugation quandle; it is a knot invariant for all counts of the homomorphisms of the conjugation quandle over any dihedral group.

\section{Quandles and Link Colorings}

\begin{definition}
A quandle~\cite[p.~105]{kamada2002} is a set $X$ with a binary operation
\[
\triangleright:X\times X\longrightarrow X
\]
with the following axioms:
\begin{enumerate}
    \item[(i)] $x\triangleright x=x$ for every $x\in X$;
    \item[(ii)] for each $y\in X$, the map $S_y(x)=x\triangleright y$ is a bijection;
    \item[(iii)] $(x\triangleright y)\triangleright z=(x\triangleright z)\triangleright(y\triangleright z)$ for all $x,y,z\in X$.
\end{enumerate}
\end{definition}

The fundamental quandle $Q(L)$ of a link $L$ is formed by the arcs and crossings of a link diagram~\cite[Sections~1--3]{kamada2002}. A quandle relation occurs at every crossing, and every arc in the link is colored by a quandle element. Let us suppose that $T$ is a finite quandle, then a $T$-coloring of $L$ is a quandle homomorphism $Q(L)\to T$, and the quandle counting invariant is
\[
\left|\operatorname{Hom}_{\mathrm{Qnd}}(Q(L),T)\right|.
\]
Background on quandles and counting invariants can be found in \cite[Chapters~3--4]{elhamdadi}.

For a given group $G$, the conjugation quandle~\cite[p.~105]{kamada2002} is the set $G$ with the operation
\[
x\triangleright y=y^{-1}xy.
\]

\begin{definition}
At each crossing, we set $a$ as the over-arc and order the two under-arcs as $b,c$ which makes the quandle relation~\cite[pp.~103--105]{kamada2002}
\[
c=b\triangleright a=a^{-1}ba.
\]
We call $b$ the source under-arc and $c$ the target under-arc of the relation.
\end{definition}

For the forward quandle, the source under-arc and target under arc are oriented so that
\[
\text{target}=\text{source}\triangleright\text{overarc}.
\]
For the inverse crossing which corresponds to the inverse quandle, their roles are reversed. If we always label the under-arcs geometrically as $b_{\mathrm{in}}$ and $b_{\mathrm{out}}$, then one crossing type uses
\[
b_{\mathrm{out}}=b_{\mathrm{in}}\triangleright a,
\]
on the other hand, the opposite type uses
\[
b_{\mathrm{out}}=b_{\mathrm{in}}\triangleright^{-1}a.
\]
Ordering to make sure that each arc is in the forward relation allows the same quandle to be used at every crossing.

\section{The Dihedral Group and Component Signatures}
\label{sec:dihedral-group}

For the remainder of the paper, let $n\geq3$ and let $D_n$ represent the dihedral group with the order $2n$:
\[
D_n=\langle r,s\mid r^n=e,\ s^2=e,\ srs=r^{-1}\rangle.
\]
Every element can then be written as
\[
s^\epsilon r^x,
\qquad
\epsilon\in\{0,1\},
\qquad
x\in\Z/n\Z,
\]
where $\epsilon=0$ represents a rotation and $\epsilon=1$ represents a reflection. This means that we can use the presentation~\cite[Sections~1.2 and~2.3.2]{Daugulis2017}
\[
D_n\cong\Z_n\rtimes\Z_2.
\]
The allowed multiplication operations are
\[
r^ir^j=r^{i+j},
\qquad
sr^ir^j=sr^{i+j},
\qquad
r^isr^j=sr^{j-i},
\qquad
sr^isr^j=r^{j-i},
\]
and the inverse rules are
\[
(r^i)^{-1}=r^{-i},
\qquad
(sr^i)^{-1}=sr^i.
\]

We can let the source under-arc be colored by $s^\epsilon r^x$ and the over-arc by $s^\delta r^y$. Calculation using the allowed multiplication operations gives that~\cite[Lemma~2.1]{IIMS2025}
\begin{equation}
(\epsilon,x)\triangleright(\delta,y)
=
\left(\epsilon,(-1)^\delta x+2\epsilon y\right).
\label{eq:dihedral-action}
\end{equation}
The binary component $\epsilon$ does not change under these algebraic operations.

Let $T_2=\{0,1\}$ be the trivial quandle, with operation $\epsilon\triangleright\delta=\epsilon$~\cite[p.~105]{kamada2002}. Equation~\eqref{eq:dihedral-action} shows that
\[
\pi:\Conj(D_n)\longrightarrow T_2,
\qquad
\pi(s^\epsilon r^x)=\epsilon,
\]
is a quandle homomorphism.

\begin{definition}
Let $L=L_1\cup\cdots\cup L_\ell$ be an oriented link. The component signature is a vector
\[
\tau=(\tau_1,\ldots,\tau_\ell)\in\{0,1\}^\ell,
\]
where $\tau_j=0$ means that the $j$th component is colored by rotations and $\tau_j=1$ means that it is colored by reflections.
\end{definition}

\begin{proposition}
The rotation-reflection type of a $\Conj(D_n)$-coloring is unchanged on each link component~\cite[Lemma~2.1(1) and Remark~4]{IIMS2025}. Additionally, every $\tau\in\{0,1\}^\ell$ occurs, so an $\ell$-component link has exactly $2^\ell$ possible component signatures.
\end{proposition}

\begin{proof}
Let $f:Q(L)\to\Conj(D_n)$ be a coloring. The relation $Q(L)\to T_2$ records whether each arc is colored by a rotation or reflection. At every crossing, the two under-arcs have the same type. This causes the type to propagate around each component, making every arc on the link component of the same type, either a reflection or rotation.

Let us fix $\tau\in\{0,1\}^\ell$ and assign the exponent $0$ to every arc in the link. Color a component by $e$ when the signature entry is $0$ and by $s$ when its signature entry is $1$. Equation~\eqref{eq:dihedral-action} shows that every crossing relation is satisfied. This means that every component signature occurs.
\end{proof}

\section{Local Relations and Signature Matrices}

Let $x_a,x_b,x_c$ be the exponents of the over-arc, source under-arc, and target under-arc at any crossing. If the over-arc has exponents with type $\delta$ and the under-arcs have type $\epsilon$, Equation~\eqref{eq:dihedral-action} gives that
\begin{equation}
x_c\equiv(-1)^\delta x_b+2\epsilon x_a\pmod n.
\label{eq:local-relation}
\end{equation}
The four possible types of reflection rotation signatures are therefore~\cite[Lemma~2.1]{IIMS2025}
\[
\begin{array}{c@{\qquad}c@{\qquad}c}
\toprule
(\delta,\epsilon,\epsilon) & \text{Exponent formula} & \text{Integer row relation}\\
\midrule
(0,0,0) & x_c=x_b & x_b-x_c=0\\
(1,0,0) & x_c=-x_b & x_b+x_c=0\\
(0,1,1) & x_c=x_b+2x_a & 2x_a+x_b-x_c=0\\
(1,1,1) & x_c=-x_b+2x_a & 2x_a-x_b-x_c=0.\\
\bottomrule
\end{array}
\]

We can then note that $x_1=x_b$ for the source under-arc, $x_2=x_a$ for the over-arc, and $x_3=x_c$ for the target under-arc, with the types $[\epsilon_2,\epsilon_1,\epsilon_3]$, these cases combine to give a relation that will hold true for any crossing
\begin{equation}
2\epsilon_1x_2+(-1)^{\epsilon_2}x_1-x_3=0.
\label{eq:universal-crossing}
\end{equation}
For a fixed component signature, each $\epsilon_i$ is determined by the component which contains the corresponding arc.

\begin{remark}[The Fox-coloring]
For the all-reflection signature, $\epsilon=\delta=1$, Equation~\eqref{eq:dihedral-action} becomes
\[
x\triangleright y=2y-x.
\]
This means that the reflection subquandle of $\Conj(D_n)$ is simply the standard dihedral quandle $R_n$~\cite[p.~105]{kamada2002}. This dihedral quandle has the crossing relations
\[
2x_a\equiv x_b+x_c\pmod n.
\]
The all-reflection signature therefore is the same as the Fox $n$-coloring system~\cite[Remark~2]{IIMS2025}~\cite[Section~2.1]{kauffman}. All-rotation and all-reflection signatures are both contained within the full family of component signatures. However, the full family of component signatures also contains mixed signatures, where the type of the over-arc is different from the type of the under arcs, which not occur when the arcs are only subject to the reflection subquandle.

For a knot, only the all-rotation and all-reflection signatures occur. The all-rotation sector contributes $n$ additional constant rotation colorings.
\end{remark}

For a fixed component signature $\tau$ on a link diagram $D$ with $m$ arcs and $k$ crossings, a component with no undercrossings is considered as one arc. We can then assign the exponent variables $x_1,\ldots,x_m$ to the arcs. Each crossing in the link then will contribute a row to the matrix using Equation~\eqref{eq:universal-crossing}, giving
\[
M_\tau(D)\in\operatorname{Mat}_{k\times m}(\Z).
\]
Therefore, through the exponent relations, we are given a system of equations that is solved when
\[
M_\tau(D)x\equiv0\pmod n.
\]
The matrix depends on the diagram, the arc labels, the component ordering, and the relation-oriented ordering at each crossing, but not on $n$.

\begin{proposition}
For a fixed component signature $\tau$, the $\Conj(D_n)$-colorings of a link $L$ with a signature $\tau$ are in bijection with the solutions
\[
M_\tau(D)x\equiv0\pmod n.
\]
\end{proposition}

\begin{proof}
Each arc color either has the form $r^{x_i}$ or $sr^{x_i}$ depending on the component signature. Equation~\eqref{eq:dihedral-action} shows that the quandle relation at a crossing is the same as the row of $M_\tau(D)$. This means that coloring the link gives a solution for the exponent system. Conversely, a solution assigns also a group element of a certain type to every arc and satisfies the crossing relations for a link.
\end{proof}

\section{Signature Modules and Smith Normal Form}

Let $A(D)$ be the set of arcs of a link diagram $D$. The arcs then generate an abelian group that is independent of the dihedral group and is also a $\Z$-module
\[
F(D)=\bigoplus_{a\in A(D)}\Z e_a\cong\Z^m.
\]
For a fixed signature $\tau$ of a link, let $R_\tau(D)\subseteq F(D)$ be the submodule that is generated by the integer crossing relations in Equation~\eqref{eq:universal-crossing}.  The purpose of the signature module is to keep the same integer crossing equations without choosing a modulus \(n\). Reducing these equations modulo \(n\) then recovers the colorings with the signature \(\tau\).

\begin{definition}
The signature module is
\begin{equation}
A_\tau(D)=F(D)/R_\tau(D)
\cong
\Z^m/\rowspan_{\Z}\bigl(M_\tau(D)\bigr).
\label{eq:signature-module}
\end{equation}
In this signature module, $A_\tau(D)$ is the abelian group generated by the arcs of the link diagram $D$. This abelian group obeys the rules that the integer crossing relations are determined by $\tau$.
\end{definition}

An assignment $x_1,\ldots,x_m\in\Z/n\Z$ is a coloring with the rotation-reflection signature $\tau$ if and only if it satisfies every crossing relation in the link. Therefore,
\[
\{\text{colorings with signature }\tau\}
\cong
\HomZ\bigl(A_\tau(D),\Z/n\Z\bigr).
\]

If $M_\tau(D)$ has rank $r_\tau$ and nonzero entries in the Smith form~\cite[Section~12.1]{dummitfoote}~\cite[Sections~2.1--2.2]{stanley},
\[
d_{\tau,1}\mid d_{\tau,2}\mid\cdots\mid d_{\tau,r_\tau}.
\]
then
\[
A_\tau(D)
\cong
\Z^{m-r_\tau}\oplus
\bigoplus_{i=1}^{r_\tau}\Z/d_{\tau,i}\Z,
\]
Smith entries with value of $1$ can be omitted; they do not contribute to the total number of colors. We call the free rank $m-r_\tau$, together with the nonunit Smith entries $d_{\tau,i}>1$, the reduced Smith data.

\begin{proposition}
Let $M$ be an integer matrix with $m$ columns, rank $r$, and Smith entries $d_1,\ldots,d_r$ that are greater than $0$. Then
\[
\left|\{x\in(\Z/n\Z)^m:Mx\equiv0\pmod n\}\right|
=
n^{m-r}\prod_{i=1}^{r}\gcd(d_i,n).
\]
\end{proposition}

\begin{proof}
The row and column operations that are used to obtain the Smith normal form are invertible under the modulus $n$, so they do not change the number of solutions, but just gives a different representation for the matrix. Each of the columns that only contain $0$ gives one free variable and contributes a factor $n$. For a nonzero Smith entry $d$, the equality $dy\equiv0\pmod n$ then has $\gcd(d,n)$ solutions, and multiplying these solutions together gives the formula for the Smith signatures~\cite[Corollary~2.2]{ge}.
\end{proof}

\section{Reidemeister Invariance of the Signature Modules}

\begin{definition}
Two integer matrices are called equivalent, written $M\sim N$, if one can be obtained from the other by adding or subtracting an integer multiple of a row from another, adding or subtracting an integer multiple of a column from another, multiplying a row or column by $\pm1$, and adjoining or deleting an isolated entry $\pm1$ together with its row and column. We also allow adjoining or deleting a zero row. A zero row represents the relation $0=0$, so this
operation does not change the abelian group that is represented by the matrix.~\cite[Section~2.1]{ge}.
\end{definition}

Row operations will replace the crossing relations with an equivalent set of relations, while column operations will replace the arc variables with an equivalent set of variables. An isolated entry both rowwise and columnwise $\pm1$ corresponds to a variable that is completely determined by one relation, so that variable and relation may be removed together. Therefore, these operations do not change the abelian group that is presented by the matrix~\cite[Theorem~2.3]{stanley}.

These operations preserve the free rank and the nonunit entries in the Smith normal form of the matrix~\cite[Sections~2.1--2.2]{stanley}. Therefore, the number of unit entries do not need to be preserved, since an identity block may be added or removed. For this reason, the invariant is the signature module, which is just its reduced Smith data. It is not the complete Smith diagonal.

We now define the two geometric crossing types. Let $y$ be the exponent of the over-arc, let $x_{\mathrm{in}}$ and $x_{\mathrm{out}}$ be the incoming and outgoing under-arc exponents, and let
\[
\delta,\epsilon\in\{0,1\}
\]
be the types of the over- and under-components. Set
\[
u=(-1)^\delta.
\]
Equation~\eqref{eq:dihedral-action} gives
\[
x_{\mathrm{out}}=ux_{\mathrm{in}}+2\epsilon y
\]
for the regular operation.

For the inverse operation, suppose that
\[
z=ux+2\epsilon y
\]
is the result of the regular operation. Solving this equation for $x$, and using $u^2=1$, gives
\[
x=uz-2u\epsilon y.
\]
Therefore, after renaming $z$ as $x_{\mathrm{in}}$ and $x$ as $x_{\mathrm{out}}$, the inverse operation is
\[
x_{\mathrm{out}}=ux_{\mathrm{in}}-2u\epsilon y.
\]

\begin{table}[H]
\centering
\caption{Regular and inverse crossing relations in exponent form.}
\small
\resizebox{\textwidth}{!}{%
\begin{tabular}{ccccc}
\toprule
$(\delta,\epsilon)$ & Regular crossing & Inverse crossing & Regular row & Inverse row\\
\midrule
$(0,0)$ & $x_{\mathrm{out}}=x_{\mathrm{in}}$ & $x_{\mathrm{out}}=x_{\mathrm{in}}$ & $x_{\mathrm{in}}-x_{\mathrm{out}}=0$ & $x_{\mathrm{out}}-x_{\mathrm{in}}=0$\\
$(1,0)$ & $x_{\mathrm{out}}=-x_{\mathrm{in}}$ & $x_{\mathrm{out}}=-x_{\mathrm{in}}$ & $-x_{\mathrm{in}}-x_{\mathrm{out}}=0$ & $-x_{\mathrm{out}}-x_{\mathrm{in}}=0$\\
$(0,1)$ & $x_{\mathrm{out}}=x_{\mathrm{in}}+2y$ & $x_{\mathrm{out}}=x_{\mathrm{in}}-2y$ & $2y+x_{\mathrm{in}}-x_{\mathrm{out}}=0$ & $2y+x_{\mathrm{out}}-x_{\mathrm{in}}=0$\\
$(1,1)$ & $x_{\mathrm{out}}=-x_{\mathrm{in}}+2y$ & $x_{\mathrm{out}}=-x_{\mathrm{in}}+2y$ & $2y-x_{\mathrm{in}}-x_{\mathrm{out}}=0$ & $2y-x_{\mathrm{out}}-x_{\mathrm{in}}=0$\\
\bottomrule
\end{tabular}%
}
\end{table}

The table shows that $\triangleright$ and $\triangleright^{-1}$ are different for the mixed case $(\delta,\epsilon)=(0,1)$. For the oriented crossing conventions that we determined before, we can set the source $x_s$ and target $x_t$ so that $x_t=x_s\triangleright y$. At a regular crossing $(x_s,x_t)=(x_{\mathrm{in}},x_{\mathrm{out}})$, while at an inverse crossing $(x_s,x_t)=(x_{\mathrm{out}},x_{\mathrm{in}})$. Both cases can therefore be represented by the equation
\begin{equation}
2\epsilon y+(-1)^\delta x_s-x_t=0.
\label{eq:oriented-crossing}
\end{equation}

Polyak proved that the four oriented moves $\Omega_{1a}$, $\Omega_{1b}$, $\Omega_{2a}$, $\Omega_{3a}$ generate all of the oriented Reidemeister moves \cite[Theorem~1.1]{polyak}. We confirm that the reduced Smith data is invariant under these moves. Because every matrix operation below is invertible, each calculation also proves the invariance of the corresponding move in the reverse direction.

\begin{lemma}[Splitting an arc by an equality relation]
Suppose one arc is replaced by two generators $x_1,x_2$ with the relation $x_1-x_2=0$. If $M_{\mathrm{before}}$ and $M_{\mathrm{after}}$ are the matrices before and after the arc is split, then
\[
M_{\mathrm{after}}\sim M_{\mathrm{before}}\oplus[1].
\]
\end{lemma}

\begin{proof}
The two new arc generators can both occur in crossing relations outside the local crossing.
\[
M_{\mathrm{after}}=
\begin{bmatrix}
A&v&w\\
0&1&-1
\end{bmatrix},
\]
In the matrix above, $A$ represents the other columns from nonlocal arcs and $v,w$ of $x_1,x_2$ in other nonlocal rows. Before the arc is split, the two segments form one arc, so
\[
M_{\mathrm{before}}=
\begin{bmatrix}
A&v+w
\end{bmatrix}.
\]
Apply the column operation $X_2\leftarrow X_2+X_1$:
\[
M_{\mathrm{after}}\sim
\begin{bmatrix}
A&v&v+w\\
0&1&0
\end{bmatrix}.
\]
For each external row $R_i$, let $v_i$ be its entry in the pivot column and apply $R_i\leftarrow R_i-v_iR_c$, where $R_c$ is the equality row. This cancels out the pivot column and simply gives
\[
M_{\mathrm{after}}\sim
\begin{bmatrix}
A&0&v+w\\
0&1&0
\end{bmatrix}
\sim M_{\mathrm{before}}\oplus[1].
\]
\end{proof}

\begin{figure}[H]
\centering
\begin{tikzpicture}[line width=0.9pt,>=Stealth]
    \begin{scope}[xshift=-4.2cm]
        \draw[->] (-1.9,-1.05) -- (-1.9,1.15);
        \draw[<->] (-1.5,0) -- node[above] {$\Omega_{1a}$} (-0.55,0);
        \draw
            (0,-1.05)
            .. controls (0,-0.35) and (0.05,-0.12) .. (0.15,0.05)
            .. controls (0.20,0.12) and (0.25,0.18) .. (0.30,0.24)
            .. controls (0.36,0.31) and (0.42,0.38) .. (0.48,0.43)
            .. controls (0.75,0.70) and (1.05,0.70) .. (1.18,0.45)
            .. controls (1.36,0.12) and (1.14,-0.23) .. (0.80,-0.20)
            .. controls (0.58,-0.18) and (0.44,0.02) .. (0.38,0.14);
        \draw[->]
            (0.22,0.34)
            .. controls (0.10,0.70) and (0.00,0.93) .. (-0.12,1.15);
        \node[left] at (-0.02,-0.58) {$x_1$};
        \node[left] at (-0.10,0.91) {$x_2$};
    \end{scope}

    \begin{scope}[xshift=3.5cm]
        \draw[<-] (-1.9,-1.05) -- (-1.9,1.15);
        \draw[<->] (-1.5,0) -- node[above] {$\Omega_{1b}$} (-0.55,0);
        \draw[<-]
            (0,-1.05)
            .. controls (0,-0.35) and (0.05,-0.12) .. (0.15,0.05)
            .. controls (0.20,0.12) and (0.25,0.18) .. (0.30,0.24)
            .. controls (0.36,0.31) and (0.42,0.38) .. (0.48,0.43)
            .. controls (0.75,0.70) and (1.05,0.70) .. (1.18,0.45)
            .. controls (1.36,0.12) and (1.14,-0.23) .. (0.80,-0.20)
            .. controls (0.58,-0.18) and (0.44,0.02) .. (0.38,0.14);
        \draw
            (0.22,0.34)
            .. controls (0.10,0.70) and (0.00,0.93) .. (-0.12,1.15);
        \node[left] at (-0.02,-0.58) {$x_1$};
        \node[left] at (-0.10,0.91) {$x_2$};
    \end{scope}
\end{tikzpicture}
\caption{The two oriented type-I generators in Polyak's generating set.}
\label{fig:r1}
\end{figure}
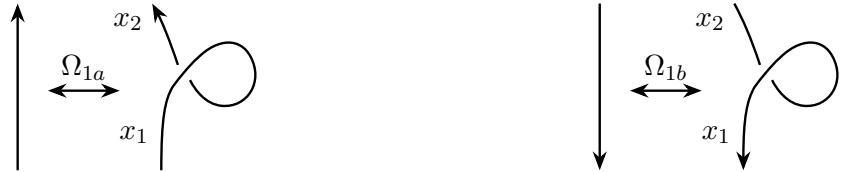

\subsection{The move \texorpdfstring{$\Omega_{1a}$}{Omega 1a}}

The two arcs in Figure~\ref{fig:r1} lie on the same component, so they both have the same signature type $\epsilon\in\{0,1\}$. In $\Omega_{1a}$, the over-arc is $x_1$, and the relation-oriented source and target under-arcs are $x_1$ and $x_2$, respectively. This means that the crossing relation is
\[
2\epsilon x_1+(-1)^\epsilon x_1-x_2=0.
\]
The identity
\[
2\epsilon+(-1)^\epsilon=1,
\qquad
\epsilon\in\{0,1\},
\]
reduces this relation to
\[
x_1-x_2=0.
\]
Thus, the move only splits one arc into two and adds an equality relation. Lemma~6.2 gives
\[
M_{\mathrm{after}}\sim M_{\mathrm{before}}\oplus[1].
\]

\subsection{The move \texorpdfstring{$\Omega_{1b}$}{Omega 1b}}

The two arcs in Figure~\ref{fig:r1} lie on the same component, so they both have the same signature type $\epsilon\in\{0,1\}$. In $\Omega_{1b}$, the over-arc is $x_1$, and the relation-oriented source and target under-arcs are $x_2$ and $x_1$, respectively. This means that the crossing relation is
\[
2\epsilon x_1+(-1)^\epsilon x_2-x_1=0.
\]
The identity
\[
2\epsilon-1=-(-1)^\epsilon,
\qquad
\epsilon\in\{0,1\},
\]
reduces this relation to
\[
(-1)^\epsilon(x_2-x_1)=0.
\]
Since $(-1)^\epsilon=\pm1$, multiplying by this factor reduces the relation to
\[
x_1-x_2=0.
\]
Thus, the move only splits one arc into two and adds one equality relation. Lemma~6.2 gives
\[
M_{\mathrm{after}}\sim M_{\mathrm{before}}\oplus[1].
\]

\subsection{The move \texorpdfstring{$\Omega_{2a}$}{Omega 2a}}

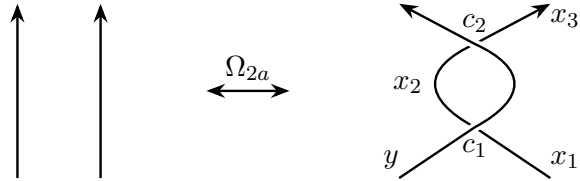
\begin{figure}[H]
\centering
\begin{tikzpicture}[line width=0.9pt,>=Stealth]
    \begin{scope}[xshift=-3.5cm]
        \draw[->] (-0.55,-1.15) -- (-0.55,1.15);
        \draw[->] (0.55,-1.15) -- (0.55,1.15);
    \end{scope}
    \draw[<->] (-1.55,0) -- node[above] {$\Omega_{2a}$} (-0.45,0);
    \begin{scope}[xshift=2.0cm]
        \draw[->]
            (1.0,-1.15) -- (0,-0.48)
            .. controls (-0.70,-0.10) and (-0.70,0.30) .. (0,0.62)
            -- (1.0,1.15);
        \draw[white,line width=3.5pt]
            (-1.0,-1.15) -- (0,-0.48)
            .. controls (0.70,-0.10) and (0.70,0.30) .. (0,0.62)
            -- (-1.0,1.15);
        \draw[->]
            (-1.0,-1.15) -- (0,-0.48)
            .. controls (0.70,-0.10) and (0.70,0.30) .. (0,0.62)
            -- (-1.0,1.15);
        \node[right] at (0.86,-0.91) {$x_1$};
        \node[left] at (-0.55,0.08) {$x_2$};
        \node[right] at (0.86,0.96) {$x_3$};
        \node[left] at (-0.86,-0.91) {$y$};
        \node[below] at (0,-0.50) {$c_1$};
        \node[above] at (0,0.64) {$c_2$};
    \end{scope}
\end{tikzpicture}
\caption{Polyak's oriented move $\Omega_{2a}$. Both strands are oriented upward. The strand labeled $y$ passes over at both crossings.}
\label{fig:r2}
\end{figure}

Let the over-strand have exponent $y$ as well as type $\delta$, and let the under-strand have type $\epsilon$. The two crossings have opposite types, meaning that one has a forward quandle relation, while the other one has an inverse quandle relation. When we use the oriented convention, which differentiates between the forward and inverse convention, their triples are
\[
c_1=(y,x_1,x_2),
\qquad
c_2=(y,x_3,x_2).
\]
The lower crossing then gives the equation
\[
2\epsilon y+(-1)^\delta x_1-x_2=0,
\]
while the upper crossing gives the equation
\[
2\epsilon y+(-1)^\delta x_3-x_2=0.
\]
We set $a=(-1)^\delta$. The matrix after the move $\Omega_{2a}$ is then
\[
M_{\mathrm{after}}=
\begin{array}{c|cccc}
&y&x_1&x_2&x_3\\ \hline
c_1&2\epsilon&a&-1&0\\
c_2&2\epsilon&0&-1&a
\end{array}.
\]
Apply $c_2\leftarrow c_2-c_1$:
\[
M_{\mathrm{after}}\sim
\begin{array}{c|cccc}
&y&x_1&x_2&x_3\\ \hline
c_1&2\epsilon&a&-1&0\\
c_2&0&-a&0&a
\end{array}.
\]
The $x_2$-column is local to the move, meaning that it does not occur in any crossings outside of the move, and it has the entry $-1$ in row $c_1$. Using the $x_2$ column as a pivot, column operations clear the other entries of $c_1$, so the pair $(c_1,x_2)$ is one identity block and it splits off. The remaining relation is then
\[
a(-x_1+x_3)=0.
\]
Since $a=\pm1$, this means that the relation is $x_1-x_3=0$. It identifies the two external pieces of the under-strand, and completely recovers the single under-arc that existed before the move. By Lemma~6.2, this equality only contributes a second identity block. Therefore
\[
M_{\mathrm{after}}\sim M_{\mathrm{before}}\oplus I_2,
\]
so the signature module is unchanged under $\Omega_{2a}$.

\subsection{The move \texorpdfstring{$\Omega_{3a}$}{Omega 3a}}

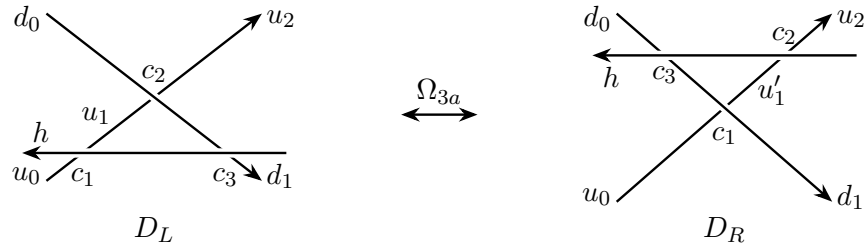
\begin{figure}[H]
\centering
\begin{tikzpicture}[line width=0.9pt,>=Stealth,scale=0.92]
    \begin{scope}[xshift=-4.1cm]
        \draw[->] (-1.55,-0.95) -- (1.55,1.45);
        \draw[white,line width=3.5pt] (-1.55,1.45) -- (1.55,-0.95);
        \draw[->] (-1.55,1.45) -- (1.55,-0.95);
        \draw[white,line width=3.5pt] (-1.9,-0.55) -- (1.9,-0.55);
        \draw[<-] (-1.9,-0.55) -- (1.9,-0.55);
        \node[left] at (-1.47,1.38) {$d_0$};
        \node[right] at (1.46,-0.88) {$d_1$};
        \node[left] at (-1.47,-0.88) {$u_0$};
        \node[right] at (1.46,1.38) {$u_2$};
        \node[above] at (-1.63,-0.55) {$h$};
        \node[left] at (-0.47,-0.02) {$u_1$};
        \node[below] at (-1.02,-0.62) {$c_1$};
        \node[above] at (0,0.31) {$c_2$};
        \node[below] at (1.02,-0.62) {$c_3$};
        \node at (0,-1.65) {$D_L$};
    \end{scope}

    \draw[<->] (-0.55,0) -- node[above] {$\Omega_{3a}$} (0.55,0);

    \begin{scope}[xshift=4.1cm]
        \draw[->] (-1.55,-1.25) -- (1.55,1.45);
        \draw[white,line width=3.5pt] (-1.55,1.45) -- (1.55,-1.25);
        \draw[->] (-1.55,1.45) -- (1.55,-1.25);
        \draw[white,line width=3.5pt] (-1.9,0.85) -- (1.9,0.85);
        \draw[<-] (-1.9,0.85) -- (1.9,0.85);
        \node[left] at (-1.47,1.38) {$d_0$};
        \node[right] at (1.46,-1.18) {$d_1$};
        \node[left] at (-1.47,-1.18) {$u_0$};
        \node[right] at (1.46,1.38) {$u_2$};
        \node[below] at (-1.63,0.85) {$h$};
        \node[right] at (0.32,0.35) {$u'_1$};
        \node[below] at (0,-0.02) {$c_1$};
        \node[above] at (0.86,0.85) {$c_2$};
        \node[below] at (-0.86,0.85) {$c_3$};
        \node at (0,-1.65) {$D_R$};
    \end{scope}
\end{tikzpicture}
\caption{Polyak's oriented move $\Omega_{3a}$ with the arc labels used in the proof. The horizontal strand is oriented from right to left and passes over both diagonal strands. The descending strand passes over the rising strand.}
\label{fig:r3}
\end{figure}

There are three strands in Figure~\ref{fig:r3}. Let the horizontal, descending, and rising strands have types
\[
\delta,\qquad\eta,\qquad\epsilon,
\]
respectively. Their types are constant before and after the move because each of the strands lies on one component. Set
\[
a=(-1)^\delta,
\qquad
b=(-1)^\eta.
\]
The horizontal strand has the exponent $h$. The descending strand has the exponents $d_0,d_1$ on the two sides of its undercrossing with the horizontal strand. The rising strand has external exponents $u_0$ and $u_2$. It also has one internal exponent on each side of the move: $u_1$ on $D_L$ and $u'_1$ on $D_R$.

The relation-oriented triples are
\[
\begin{array}{c|c|c}
\toprule
\text{Side}&\text{Crossing}&(\text{over},\text{source},\text{target})\\
\midrule
D_L&c_1&(h,u_1,u_0)\\
D_L&c_2&(d_0,u_1,u_2)\\
D_L&c_3&(h,d_0,d_1)\\
D_R&c_1&(d_1,u_0,u'_1)\\
D_R&c_2&(h,u_2,u'_1)\\
D_R&c_3&(h,d_0,d_1)\\
\bottomrule
\end{array}
\]
The reversed order of the source-target strand at the horizontal-rising crossings shows the inverse geometric crossing relation, therefore, we use the $\triangleright^{-1}$ relation.

On the left side, the three crossing relations are
\[
2\epsilon h+au_1-u_0=0,
\]
\[
2\epsilon d_0+bu_1-u_2=0,
\]
and
\[
2\eta h+ad_0-d_1=0.
\]
With the columns ordered as $h,d_0,d_1,u_0,u_1,u_2$, the local matrix is then
\[
M_L=
\begin{array}{c|cccccc}
&h&d_0&d_1&u_0&u_1&u_2\\ \hline
c_1&2\epsilon&0&0&-1&a&0\\
c_2&0&2\epsilon&0&0&b&-1\\
c_3&2\eta&a&-1&0&0&0
\end{array}.
\]
Using the integer matrix operations, we multiply $c_1$ by the $a$, and then use the operation $c_2\leftarrow c_2-bc_1$ to get the second row. The second row then becomes
\[
(-2ab\epsilon,\ 2\epsilon,\ 0,\ ab,\ 0,\ -1).
\]
The $u_1$-column has the elements $(1,0,0)$; it does not interact with any crossings outside of the link. Using this column as a pivot, we can clear the other entries in the row $c_1$ using column operations, and delete the identity block that is a result from this operation. The reduced matrix with only the crossings is then
\[
N_L=
\begin{array}{c|ccccc}
&h&d_0&d_1&u_0&u_2\\ \hline
c_2&-2ab\epsilon&2\epsilon&0&ab&-1\\
c_3&2\eta&a&-1&0&0
\end{array}.
\]

On the right side, the crossing relations are
\[
2\epsilon d_1+bu_0-u'_1=0,
\]
\[
2\epsilon h+au_2-u'_1=0,
\]
and
\[
2\eta h+ad_0-d_1=0.
\]
Thus
\[
M_R=
\begin{array}{c|cccccc}
&h&d_0&d_1&u_0&u'_1&u_2\\ \hline
c_1&0&0&2\epsilon&b&-1&0\\
c_2&2\epsilon&0&0&0&-1&a\\
c_3&2\eta&a&-1&0&0&0
\end{array}.
\]
Using the integer matrix operations,
\[
c_1\leftarrow c_1-c_2,
\qquad
c_1\leftarrow ac_1,
\qquad
c_1\leftarrow c_1+2a\epsilon c_3.
\]
The first row that results from this is then
\[
\bigl(2a\epsilon(2\eta-1),\ 2\epsilon,\ 0,\ ab,\ 0,\ -1\bigr).
\]
Since $b=(-1)^\eta=1-2\eta$, we have $2\eta-1=-b$. Therefore this row is
\[
(-2ab\epsilon,\ 2\epsilon,\ 0,\ ab,\ 0,\ -1).
\]
The $u'_1$-column now equals $(0,-1,0)$. This makes it a unit pivot, so we use row $c_2$ and then can delete the $u'_1$-column because it is an identity block. The remaining matrix is
\[
N_R=
\begin{array}{c|ccccc}
&h&d_0&d_1&u_0&u_2\\ \hline
c_1&-2ab\epsilon&2\epsilon&0&ab&-1\\
c_3&2\eta&a&-1&0&0
\end{array}.
\]
Therefore $N_R=N_L$. Both sides reduce to the same matrix, so
\[
M_L\sim M_R.
\]
We have proven that the signature module is unchanged under the move $\Omega_{3a}$.

\begin{theorem}[Signature-module invariance]
Let $L$ be an oriented ordered link with a fixed component with signature $\tau$. The reduced Smith data of $M_\tau(D)$ does not depend on the chosen link diagram $D$.
\end{theorem}

\begin{proof}
The calculations above prove that under Polyak's generators, the reduced Smith data does not change.
\[
\Omega_{1a},
\qquad
\Omega_{1b},
\qquad
\Omega_{2a},
\qquad
\Omega_{3a}.
\]
Polyak proved that the above moves generate the entire set of ordered, oriented Reidemeister moves \cite[Theorem~1.1]{polyak}. Therefore proving that every generator in the Polyak set does not change the Smith module also proves that the smith module is a knot invariant.
\end{proof}

Since Theorem~6.3 shows that $A_\tau(D)\cong A_\tau(D')$ for any two diagrams $D,D'$ of $L$, we can denote its isomorphism class as $A_\tau(L)$.

\section{The General Smith Normal Form Formula}

\begin{theorem}[General SNF decomposition]
Let $L=L_1\cup\cdots\cup L_\ell$ be a link that is oriented according to the conventions from before, and represented by a link diagram $D$ with $m$ arcs. For each $\tau\in\{0,1\}^\ell$, let the nonzero Smith entries of $M_\tau(D)$ be $d_{\tau,1},\ldots,d_{\tau,r_\tau}$. Then, for every $n\geq3$,
\[
\left|\HomQnd\bigl(Q(L),\Conj(D_n)\bigr)\right|
=
\sum_{\tau\in\{0,1\}^\ell}
n^{m-r_\tau}
\prod_{i=1}^{r_\tau}\gcd(d_{\tau,i},n).
\]
\end{theorem}

\begin{proof}
Each coloring determines a component signature, and gives the coloring set with $\HomZ(A_\tau(L),\Z/n\Z)$, and the Smith normal form gives the size of this set. Summing over the $2^\ell$ signatures gives the result.
\end{proof}

If more than one signature module is equivalent, and have the same nonunit Smith normal form, then contributions from that one signature module may be grouped together.

\begin{corollary}[Grouped form]
The distinct reduced Smith types are indexed by $j$. Let $a_j$ be the number of component signatures producing the $j$th type, let $\nu_j$ be its free rank, and let $e_{j,1},\ldots,e_{j,s_j}>1$ be its nonunit Smith entries. Then
\[
\left|\HomQnd\bigl(Q(L),\Conj(D_n)\bigr)\right|
=
\sum_j a_jn^{\nu_j}\prod_{i=1}^{s_j}\gcd(e_{j,i},n),
\]
where entries that only contain $1$ do not contribute anything to the final count.
\end{corollary}

The all-rotation signature contributes $n^\ell$, and the all-reflection signature gives the Fox $n$-colorings~\cite[Section~2.1]{kauffman}, while the remaining $2^\ell-2$ signatures are mixed.

\section{Mixed Rotation--Reflection Signatures Contain More Information Than Fox Colorings}

The all-reflection signature is only one part of the signature module. In this section, we show that the mixed rotation--reflection signatures can distinguish links even when the all-reflection modules are isomorphic. Thus, the full family of signature modules contains information that does not exist in Fox colorings.

We remember that a component signature is a vector
\[
\tau\in\{0,1\}^\ell,
\]
where a number $0$ is a rotation-colored component and a number $1$ is a reflection-colored component. When
\[
\tau=\mathbf{1}=(1,\ldots,1),
\]
every component is reflection colored. The crossing relations then just become the Fox coloring relations over the dihedral group~\cite[Section~2.1]{kauffman}
\[
2x_a\equiv x_b+x_c\pmod n.
\]
Consequently, the all-reflection module $A_{\mathbf{1}}(L)$ determines the number of Fox $n$-colorings for every $n$. The remaining modules $A_\tau(L)$ record the colorings in which rotations and reflections occur on different components.

\begin{proposition}
The full signature-module family
\[
\mathcal{A}(L)=\{A_\tau(L):\tau\in\{0,1\}^\ell\}
\]
contains more information than the all-reflection module $A_{\mathbf{1}}(L)$.
\end{proposition}

\begin{proof}
We look at the three-component links
\[
L_6=L6n1=6^3_3
\qquad\text{and}\qquad
L_9=L9n21=9^3_{17}.
\]
Diagrams of the two links are shown in Figure~\ref{fig:comparison-links}. The planar diagram (PD) data for these links were obtained from The Knot Atlas~\cite{knotatlas}.

\setcounter{figure}{5}
\begin{figure}[H]
\centering
\begin{minipage}[t]{0.42\textwidth}
\centering
\includegraphics[width=0.68\linewidth]{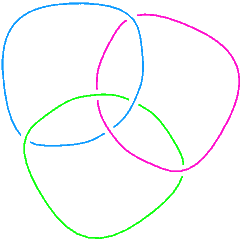}

\smallskip
(a) $L_6=L6n1$
\end{minipage}
\hfill
\begin{minipage}[t]{0.42\textwidth}
\centering
\includegraphics[width=0.68\linewidth]{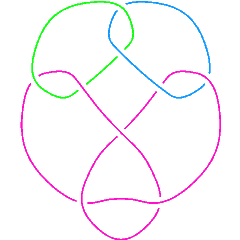}

\smallskip
(b) $L_9=L9n21$
\end{minipage}
\caption{The standard Knot Atlas diagrams of the named three-component links used in the comparison.}
\label{fig:comparison-links}
\end{figure}

These are the standard Knotscape link projections that appear in the Knot Atlas entry for L6n1 and the Knot Atlas entry for L9n21~\cite{knotatlas}.

Their signature modules are shown below.
\[
\begin{array}{c|c|c}
\tau&A_\tau(L_6)&A_\tau(L_9)\\
\hline
(0,0,0)&\Z^3&\Z^3\\
(0,0,1)&\Z\oplus\Z/2\Z\oplus\Z/2\Z&\Z\oplus\Z/2\Z\oplus\Z/6\Z\\
(0,1,0)&\Z\oplus\Z/2\Z\oplus\Z/2\Z&\Z\oplus\Z/2\Z\oplus\Z/2\Z\\
(0,1,1)&\Z^2\oplus\Z/2\Z&\Z^2\oplus\Z/2\Z\\
(1,0,0)&\Z\oplus\Z/2\Z\oplus\Z/2\Z&\Z\oplus\Z/2\Z\oplus\Z/2\Z\\
(1,0,1)&\Z^2\oplus\Z/2\Z&\Z^2\oplus\Z/2\Z\\
(1,1,0)&\Z^2\oplus\Z/2\Z&\Z^2\oplus\Z/2\Z\\
(1,1,1)&\Z\oplus\Z/2\Z\oplus\Z/2\Z&\Z\oplus\Z/2\Z\oplus\Z/2\Z
\end{array}
\]

In particular, their all-reflection modules are isomorphic:
\[
A_{(1,1,1)}(L_6)\cong A_{(1,1,1)}(L_9)
\cong\Z\oplus\Z/2\Z\oplus\Z/2\Z.
\]
Therefore, $L_6$ and $L_9$ have the same number of Fox $n$-colorings for every $n$. However, their mixed modules at $\tau=(0,0,1)$ are not isomorphic, since
\[
A_{(0,0,1)}(L_6)\cong\Z\oplus\Z/2\Z\oplus\Z/2\Z,
\]
whereas
\[
A_{(0,0,1)}(L_9)\cong\Z\oplus\Z/2\Z\oplus\Z/6\Z.
\]
All three of the signatures of $L_6$ containing exactly one reflection have the same module, while $L_9$ has one such signature containing a $\Z/6\Z$ factor. This means that the complete signature-module families of the two links are different, even though their all-reflection modules and linking matrices are the same. In particular, both links have three Fox $3$-colorings, while their full $\operatorname{Conj}(D_3)$-coloring counts
are $66$ and $72$, respectively.
\end{proof}

\section{Relation with the Multivariable Alexander Module}
\label{sec:alexander-specialization}

The two signatures that only contain either all reflections or all rotations suggest a connection with the Alexander Module~\cite[Definition~2 and Proposition~6]{traldi2020}~\cite[Definitions~4, 6, and~7]{traldi2025}. If every component is a rotation, then Equation~\eqref{eq:local-relation} becomes
\[
x_c=x_b,
\]
which is simply the trivial-quandle relation. If every component is a reflection, it then becomes
\[
x_c=2x_a-x_b,
\]
the Fox coloring relation. These are simply the $t=1$ and $t=-1$ specializations of the Alexander-quandle relation~\cite[Proposition~6]{traldi2020}
\[
x_c=tx_b+(1-t)x_a.
\]
The $t=-1$ case and its connection with Fox colorings are also described by Kauffman and Lopes~\cite[Section~2.1]{kauffman}.

The reflection-rotation colorings also follow this pattern, where the pure all-reflection signature gives the Fox coloring relation and has $\epsilon=\delta=1$. On the other hand, the pure all-rotation signature gives the trivial-quandle relation and has $\epsilon=\delta=0$. This is very similar to the Alexander coloring where $+1$ gives the trivial quandle, and $-1$ gives Fox colorings. We can then try to make this distinction in the Alexander module by observing what happens when
\[
t_k=(-1)^{\tau_k}
\]
for each component $L_k$.

Since we already know what happens for $t_k=+1$ and $t_k=-1$, we must see if the mixed signatures also match with the Alexander module. Suppose that the under-arcs $b,c$ lie on $L_i$ and the over-arc $a$ lies on $L_j$. In the notation of Traldi's Figure~1, \(a_1=a\) is the over-arc, \(a_2=b\) is the source under-arc, and \(a_3=c\) is the target under-arc. Thus \(\kappa(a_2)=i\) and \(\kappa(a_1)=j\). The corresponding relation in the Alexander module is then~\cite[Definitions~3 and 4]{traldi2018}
\begin{equation}
(1-t_i)e_a+t_je_b-e_c=0.
\label{eq:alexander-signature}
\end{equation}
After setting $t_k=(-1)^{\tau_k}$, this becomes
\[
\bigl(1-(-1)^{\tau_i}\bigr)e_a
+(-1)^{\tau_j}e_b-e_c=0.
\]
Because $1-(-1)^{\tau_i}=2\tau_i$, and because $\epsilon=\tau_i$ and $\delta=\tau_j$ at this crossing, the relation is
\[
2\epsilon e_a+(-1)^\delta e_b-e_c=0,
\]
which, from before, is the integer crossing relation that is represented by the corresponding row of $M_\tau(D)$.

\begin{theorem}
For every component signature $\tau$, the signature module $A_\tau(L)$ is obtained from the multivariable Alexander module $M_A(L)$ by setting
\[
t_k=(-1)^{\tau_k}
\]
on each component \(L_k\). Also, if
\[
\phi_\tau(t_k)=(-1)^{\tau_k},
\]
then~\cite[Lemma~11]{traldi2018}
\[
A_\tau(L)\cong
\Z_{\phi_\tau}\otimes_{\Lambda_\ell}M_A(L).
\]
\end{theorem}

\begin{proof}
Both presentations use the arcs as generators. The calculation that was done above shows that every Alexander crossing relation becomes the relation that is defined by the corresponding row of $M_\tau(D)$. Therefore the specialized Alexander presentation and the presentation of $A_\tau(D)$ are the same. The result then follows from Theorem~6.3.
\end{proof}

\section{Generalization to Generalized Dihedral Groups}

Throughout this paper, we have used the dihedral groups of regular $n$-gons. However, this restriction is not necessary. In Section~\ref{sec:dihedral-group}, we used the reflection and rotation nature of the dihedral group. However, we did not use the fact that the rotation subgroup is cyclic.

All generalized dihedral groups have the property of rotations and reflections; however, their rotation subgroups need not be cyclic. Therefore, because the quandle equations derived in Section~\ref{sec:dihedral-group} did not use this cyclicity, the same derivation applies to all generalized dihedral groups.

For an abelian group $B$, let
\[
\operatorname{Dih}(B)=B\rtimes\Z/2\Z,
\]
where the nontrivial element of $\Z/2\Z$ acts on $B$ by $x\mapsto-x$~\cite[Sections~2.3.1--2.3.2]{Daugulis2017}.

\begin{theorem}[Generalized-dihedral extension]
Let $L$ be an oriented ordered link and let $B$ be an abelian group. For every component signature $\tau$, the $\Conj(\operatorname{Dih}(B))$-colorings of $L$ with signature $\tau$ are in bijection with
\[
\HomZ\bigl(A_\tau(L),B\bigr).
\]
Therefore,
\[
\HomQnd\bigl(Q(L),\Conj(\operatorname{Dih}(B))\bigr)
\cong
\bigsqcup_{\tau\in\{0,1\}^{\ell}}
\HomZ\bigl(A_\tau(L),B\bigr).
\]
\end{theorem}

\begin{proof}
The calculation that was used to obtain the crossing equations only uses addition, additive inverses, and multiplication by $2$. These operations are defined in every abelian group $B$. Therefore, for an under-arc of type $\epsilon$ and an over-arc of type $\delta$, the conjugation calculation that was used before gives
\[
x_c=(-1)^\delta x_b+2\epsilon x_a.
\]
Equivalently,
\[
2\epsilon x_a+(-1)^\delta x_b-x_c=0.
\]
This is the same as the integer crossing relation represented by the corresponding row of $M_\tau(D)$, now interpreted in $B$.

Thus, an assignment of elements of $B$ to the arcs is a coloring with signature $\tau$ if and only if it satisfies the relations which define $A_\tau(L)$. These assignments are therefore the homomorphisms
\[
A_\tau(L)\longrightarrow B.
\]
Finally, every coloring has exactly one component signature, so taking the disjoint union over all signatures gives us the result.
\end{proof}

\section{Conclusion and Further Directions}

We use the component signatures for a quandle to turn the crossing relations into linear equations in a matrix and prove that the resulting reduced Smith data are equivalent under Reidemeister moves and ambient isotopy, and therefore are independent of the chosen diagram. Their reduced Smith data determine the $\Conj(D_n)$ counting invariant for every $n$ at once.

We were not able to determine the enhanced counting polynomial for a quandle~\cite[Definition~7]{navasnelson2008} based off of the reduced Smith data. Additional information about the link may be required in order to compute the enhanced counting polynomial. We plan to understand if additional information may be necessary, and what additional information would be enough to determine the counting polynomial.

\end{document}